\pdfoutput=1
\documentclass[letterpaper,11pt]{amsart}

\usepackage[utf8]{inputenc}
\usepackage[T1]{fontenc}
\usepackage[english]{babel}

\usepackage{graphicx}

\usepackage{amsmath}
\usepackage{amsfonts}
\usepackage{mathrsfs}
\usepackage{mathtools}

\usepackage[top=3.5cm,bottom=3.5cm,right=3cm,left=3cm]{geometry}

\usepackage{kpfonts}
\usepackage{microtype}

\usepackage[style=danish,danish=guillemets]{csquotes}

\usepackage{enumitem}

\setenumerate{label=(\alph*)}

\newif\ifarXiv
\arXivtrue
\ifarXiv \else

\usepackage[
  style=numeric,
  arxiv=abs,
  backend=biber]{biblatex}

\renewbibmacro*{doi+eprint+url}{%
  \iftoggle{bbx:doi}
  {%
    \iffieldundef{doi}
    {}
    {%
      \begingroup
      \edef\URLorDOI{%
        \detokenize{http://dx.doi.org/}%
        \thefield{doi}%
      }%
      \iffieldequals{url}{\URLorDOI}
      {\endgroup}
      {%
        \endgroup
        \printfield{doi}%
      }%  
    }%
  }
  {}%
  \newunit\newblock
  \iftoggle{bbx:eprint}
  {\usebibmacro{eprint}}
  {}%
  \newunit\newblock
  \iftoggle{bbx:url}
  {\usebibmacro{url+urldate}}
  {}}

\fi

\usepackage{url}

\usepackage[colorlinks=true]{hyperref}

\newtheorem{theorem}{Theorem}[section]
\newtheorem{lemma}[theorem]{Lemma}
\newtheorem{proposition}[theorem]{Proposition}

\newcommand{\mathsetfont}{\mathbb}
\newcommand{\DeclareMathSet}[1]{%
  \expandafter\newcommand\csname set#1\endcsname{\mathsetfont{#1}}}

\DeclareMathSet{C} % Complex numbers
\DeclareMathSet{R} % Reals
\DeclareMathSet{Z} % Integers
\DeclareMathSet{N} % Natural numbers. These include 0, by the way!

\DeclareMathOperator{\spn}{span}
\DeclareMathOperator{\id}{id}

\DeclareMathOperator{\Hom}{Hom}
\newcommand{\ab}{_{\mathrm{ab}}}

\newcommand{\boundary}{\partial}
\newcommand{\inv}{^{-1}}
\newcommand{\isom}{\mathrel{\cong}}
\newcommand{\braid}[2]{\tau_{#1}\tau_{#2}\tau_{#1}}
\newcommand{\braidrel}[2]{\braid{#1}{#2} = \braid{#2}{#1}}

\newcommand{\Twist}[1]{\tau_{#1}}
\newcommand{\Ta}[1]{\Twist{\alpha_{#1}}}

\newcommand{\Tb}[1]{\Twist{\beta_{#1}}}
\newcommand{\Tg}[1]{\Twist{\gamma_{#1}}}

\DeclarePairedDelimiterX{\Set}[2]{\lbrace}{\rbrace}{%
  #1 \;\; \vert \;\; #2}

\DeclarePairedDelimiter{\prt}{\lparen}{\rparen}

\let\phi\varphi
\let\epsilon\varepsilon

\author{Rasmus Villemoes}
\address{Centre for Quantum Geometry of Moduli Spaces\\ %
  Ny Munkegade 118, bldg. 1530\\ %
  Aarhus University\\ %
  8000 Aarhus C\\ %
  Denmark}
\curraddr{Dept. of Mathematics\\ %
  University of Maryland \\ %
  College Park, MD 20742-4015 \\ %
  United States}
\email{math@rasmusvillemoes.dk}
\thanks{Supported by QGM (Centre for Quantum Geometry of Moduli
    Spaces, funded by the Danish National Research Foundation)}
\date{\today}

\title{A computation of $H^1(\Gamma, H_1(\Sigma))$}

\begin{document}

\begin{abstract}
  Let $\Sigma = \Sigma_{g,1}$ be a compact surface of genus $g\geq 3$
  with one boundary component, $\Gamma$ its mapping class group and $M
  = H_1(\Sigma, \setZ)$ the first integral homology of $\Sigma$. Using
  that $\Gamma$ is generated by the Dehn twists in a collection of
  $2g+1$ simple closed curves (Humphries' generators) and simple
  relations between these twists, we prove that $H^1(\Gamma, M)$ is
  either trivial or isomorphic to $\setZ$. Using Wajnryb's
  presentation for $\Gamma$ in terms of the Humphries generators we
  can show that it is not trivial.
\end{abstract}

\maketitle

\section{Group cohomology in 45 seconds}
\label{sec:group-cohomology}

For $G$ a group and $M$ a (left) $G$-module (a module over the group
ring $\setZ G$), a cocycle is a map $u\colon G\to M$ satisfying the
\emph{cocycle condition},
\begin{equation}
  \label{eq:17}
  u(gh) = u(g) + gu(h),
\end{equation}
for all $g,h\in G$. The set of $M$-valued cocycles on $G$ is denoted
$Z^1(G, M)$. A coboundary is a cocycle of the form $g\mapsto
m-gm$ for some $m\in M$; the set of these is denoted $B^1(G, M)$. The
cohomology group $H^1(G, M)$ is the quotient $Z^1(G, M)/B^1(G, M)$.

It follows immediately from~\eqref{eq:17} that $u(e) = 0$ for $e$ the
identity element of $G$. It also follows that $u(g\inv) = -g\inv
u(g)$, and that
\begin{equation}
  \label{eq:31}
  u(ghg\inv) = (1-ghg\inv)u(g) + g u(h).
\end{equation}

We also note that $u$ is determined by its values on a set of
generators of $G$. Indeed, if $G = \langle g_1, \ldots, g_r \mid r_1,
\ldots, r_s\rangle$ is a finite presentation of $G$, the space of
cocycles $Z^1(G, M)$ may be identified with the subspace of $M^r$
determined by the $s$ linear equations in the $r$ unknowns $m_1 =
u(g_1), \ldots, m_r = u(g_r)$ given by the relations $r_j$ obtained by
expanding via the cocycle condition.  For example, the ($\setZ
G$-)linear equation associated to the relation $g_1g_2g_3\inv g_1 = e$
is
\begin{equation*}
  (1+g_1g_2g_3\inv)m_1 + g_1m_2 - g_1g_2g_3\inv m_3  = 0.
\end{equation*}
In the same setting, $B^1(G, M)$ may be identified with the subspace
\begin{equation*}
  \Set[\big]{ \prt[\big]{ (1-g_1)m, \ldots, (1-g_r)m } }{ m\in M } \subseteq M^r.
\end{equation*}

If the action of $G$ on $M$ is trivial, a cocycle is simply a group
homomorphism $G\to M$, and since any coboundary vanishes in this case,
we have
\begin{equation*}
  H^1(G, M) = \Hom(G, M) = \Hom(G\ab, M)
\end{equation*}
where $G\ab$ denote the abelianization of $G$.

\subsection{Exact sequences}
\label{sec:exact-sequences}

Suppose $1\to K\to G\to Q\to 1$ is a short exact squence of groups and
$M$ a $G$-module. The action of $G$ on $K$ by conjugation induces an
action on the cohomology group $H^1(K, M)$. There is an exact sequence
\begin{equation}
  \label{eq:30}
  0 \to H^1(Q, M^K) \to H^1(G, M) \to H^1(K, M)^G,
\end{equation}
where $M^K$ denote the subspace of $M$ invariant under $K$, and
$H^1(K, M)^G$ is the subspace invariant under the above-mentioned
action of $G$. This is a consequence of the Hochschild-Serre spectral
sequence, but can also be checked by direct verification using the
hands-on definitions of cocycles and coboundaries given above.

\section{Notation and conventions}
\label{sec:notation-conventions}

We consider a compact, oriented surface $\Sigma$ of genus
$g\geq 1$ with one boundary component. 
% Endowing $\Sigma$ with a fixed basepoint $*\in \Sigma$ results in
% the marked surface $\Sigma_*$; removing an open disc from $\Sigma$
% gives a surface $\Sigma_\boundary$ with one boundary component.

\subsection{Curves and homology}
\label{sec:curves-homology}

Fix a collection $\mathcal{C}$ of $3g-1$ simple, closed curves
$\alpha_j, \beta_j, \gamma_j$ as shown in
Figure~\ref{fig:curves}. With appropriate choices of orientations
(which we also fix), the homology classes $a_j = [\alpha_j]$, $b_j =
[\beta_j]$, $c_j = [\gamma_j]$ satisfy
\begin{subequations}
  \label{eq:19}
  \begin{align}
    \label{eq:1}
    \omega(a_j, a_k) = \omega(b_j, b_k) &= 0 \\
    \label{eq:2}
    \omega(a_j, b_k) &= \delta_{jk} \\
    \label{eq:3}
    c_j &= a_{j-1} - a_j
  \end{align}
\end{subequations}
where $\omega$ denotes the intersection pairing on $M = H_1(\Sigma,
\setZ)$. In particular, $S = (a_1, b_1, \ldots, a_g, b_g)$ is a
symplectic basis for $M$. We let $\mathcal{S}$ denote the subset of
$\mathcal{C}$ consisting of the $2g$ curves $\alpha_j, \beta_j$. There
are involutions on the sets $\mathcal{S}$ and $S$ given by $\alpha_j
\leftrightarrow \beta_j$ and $a_j \leftrightarrow b_j$, respectively;
we will use~$\iota$ to stand for either of these involutions. Clearly
$\iota[\eta] = [\iota \eta]$ for any $\eta\in \mathcal{S}$.

We will use the notation $M_j$ for the symplectic subspace
$\spn_\setZ(a_j, b_j)$ of $M$, and $M_j'$ for its complement
$\spn(S-\{a_j, b_j\})$. Associated to these are the projections
$\pi_j$ and $\pi_j' = \id-\pi_j$.
% Occasionally we will write e.g. $M_{j,j+1}$ for $M_j \oplus
% M_{j+1}$. Whenever we write matrices for the action of a mapping
% class on such a subspace, it is understood that it is with respect
% to the appropriate subset of $S$ (ie. $(a_j, b_j, a_{j+1}, b_{j+1})$
% in the case of $M_{j,j+1}$).

Using $\gamma_1$ and $c_1$ as synonyms for $\alpha_1$ and $a_1$,
respectively, we observe that $S' = (c_1, b_1, c_2, b_2, \ldots, c_g,
b_g)$ also constitutes a basis for $M$; this is immediate
from~\eqref{eq:3}. We use $\mathcal{S}'$ to denote the set
$\{\gamma_1, \beta_1, \ldots, \gamma_g, \beta_g\}\subset \mathcal{C}$.

\begin{figure}[ht]
  \centering
  \includegraphics{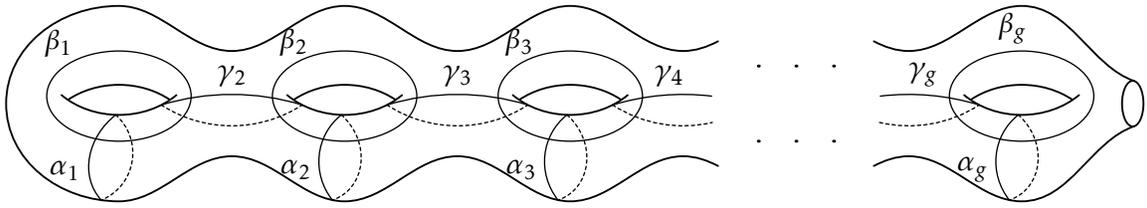}
  \caption{A collection of simple closed curves on $\Sigma$.}
  \label{fig:curves}
\end{figure}

\subsection{Twists and action}
\label{sec:twists-action}

It is well-known that the action of the (left) Dehn
twist in a simple closed curve $\eta$ on a homology element $m$ is
given by
\begin{align}
  \label{eq:4}
  \tau_\eta m = m + \omega(m, [\vec\eta])[\vec\eta],
\end{align}
where $\vec\eta$ denotes any of the oriented versions of $\eta$ (see
e.g.~\cite{MCGprimer}). For convenience, we record these consequences:
\begin{subequations}
  \label{eq:8}
  \begin{align}
    \label{eq:5}
    \tau_{\alpha_j} a_k &= a_k &
    \tau_{\beta_j} b_k &= b_k \\
    \label{eq:6}
    \tau_{\alpha_j} b_k &= b_k &
    \tau_{\beta_j} a_k &= a_k \\
    \label{eq:7}
    \tau_{\alpha_j} b_j &= b_j - a_j &
    \tau_{\beta_j} a_j &= a_j + b_j
  \end{align}
\end{subequations}
Here,~\eqref{eq:5} holds for any $1\leq j,k\leq g$, while~\eqref{eq:6}
holds for $j\not= k$. For $1 < j \leq g$ we also have
\begin{subequations}
  \label{eq:22}
  \begin{align}
    \label{eq:9}
    \tau_{\gamma_j} b_j &= b_j + a_{j-1} - a_j
                     = b_j + c_j &
    \tau_{\gamma_j} b_{j-1} &= b_{j-1} - a_{j-1} + a_j 
                        = b_{j-1} - c_j         \\
    \label{eq:10}
    \tau_{\gamma_j} a_k &= a_k \\
    \label{eq:11}
    \tau_{\gamma_j} b_k &= b_k, 
  \end{align}
\end{subequations}
with~\eqref{eq:10} holding for any $k$, and~\eqref{eq:11} for $k \not=
j-1,j$.

\subsection{Wajnryb's presentation for the mapping class group}
\label{sec:wajnrybs-pres-mapp}

The curves in $\mathcal{C}$ are not completely arbitrary; the twists
in these curves are the so-called Lickorish generators for the mapping
class group. Humphries~\cite{MR547453} showed that the $2g+1$ twists
in $\mathcal{S}' \cup \{\alpha_2\}$ actually suffice (and that $2g+1$
is the minimal number of twists needed to generate the mapping class
group). Later, Wajnryb~\cite{MR719117} was able to give a finite
presentation of the mapping glass group using Humphries generators.
% A well-known finite presentation of the mapping class group is due to
% Wajnryb~\cite{MR719117}. 
An exposition of this, along with more details on the history of
genarating and presenting the mapping class group, can be found
in~\cite{MCGprimer}. We give a slightly modified version of their
Theorem~5.3.
\begin{theorem}
  \label{thm:2}
  The mapping class group $\Gamma_{g,1}$ has a presentation with a
  generator $g_\eta$ for each curve $\eta\in \mathcal{S}\cup
  \{\alpha_2\} = \{\gamma_1, \beta_1, \gamma_2, \beta_2, \ldots,
  \gamma_g, \beta_g, \alpha_2\}$, and relations:
  \begin{enumerate}
  \item If $\eta$ and $\lambda$ are disjoint, $g_\eta$ and $g_\lambda$ commute.
  \item If $\eta$ and $\lambda$ intersect in exactly one point,
    $g_\eta g_\lambda g_\eta = g_\lambda g_\eta g_\lambda$.
  \item Let $w$ denote the word $g_{\beta_2} g_{\gamma_2} g_{\beta_1} g_{\gamma_1}
    g_{\gamma_1} g_{\beta_1} g_{\gamma_2} g_{\beta_2}$. Then
    \begin{equation}
      \label{eq:25}
      (g_{\gamma_1} g_{\beta_1} g_{\gamma_2})^4 = g_{\alpha_2} w
      g_{\alpha_2} w\inv.
    \end{equation}
  \item Let
    \begin{align*}
      w_1 &= g_{\beta_2} g_{\gamma_3} g_{\gamma_2} g_{\beta_2} &
      x_1 &= w_1\inv g_{\alpha_2} w_1 \\
      w_2 &= g_{\beta_1} g_{\gamma_2} g_{\gamma_1} g_{\beta_1} &
      x_2 &= w_2\inv x_1 w_2 \\
      w_3 &= g_{\beta_3} g_{\gamma_3} &
      x_3 &= w_3\inv x_1 w_3
      \shortintertext{and}
      w_4 &= g_{\beta_3} g_{\gamma_3} g_{\beta_2} g_{\gamma_2}
      g_{\beta_1} x_3 g_{\gamma_1}\inv g_{\beta_1}\inv g_{\gamma_2}\inv
      g_{\beta_2}\inv &
      x_4 &= w_4 g_{\alpha_2} w_4\inv.
    \end{align*}
    Then
    \begin{align}
      \label{eq:28}
      g_{\alpha_2} x_2 x_1 = g_{\gamma_1} g_{\gamma_2} g_{\gamma_3} x_4.
    \end{align}
  \end{enumerate}  
\end{theorem}
Of course, the abstract generator $g_\eta$ simply corresponds to the
twist $\tau_\eta$.

\section{Computing cohomology}
\label{sec:computing-cohomology}

In this section we compute $H^1(\Gamma, M)$. We will do this by
proving that any cocycle is cohomologous to one with some very nice
properties. This will be done in two steps. In the first, we adapt the
cocycle to the basis $S$ (or rather, to the set of curves
$\mathcal{S}$), and in the second, we add another coboundary to adapt
the cocycle to the basis $S'$ for $M$ (again, this is with respect to
the set $\mathcal{S}'$). We do not specify what it means to »adapt a
cocycle to a basis« or »adapt a cocycle to a set of curves«; it will
be apparent from the statements of Proposition~\ref{prop:3} and
Proposition~\ref{prop:5}.

% To be precise, the adaptation is not done with respect to a basis for
% $M$, but with respect to \emph{a set of oriented curves representing
%   the basis}; 

\subsection{Adapting to \texorpdfstring{$\mathcal{S}$}{S}}
\label{sec:adapting-s}

\begin{proposition}
  \label{prop:3}
  Any cohomology class in $H^1(\Gamma, M)$ is represented by
  a cocycle $u$ satisfying that for each $\eta\in \mathcal{S}$, the
  coefficient of $[\eta]$ in the expansion of $u(\tau_\eta)$ in terms
  of the basis $S$ is $0$. Moreover, this determines $u$ uniquely.
\end{proposition}
\begin{proof}
  Recall that $\mathcal{S}$ is the collection of the $2g$ curves
  $\alpha_j, \beta_j$, the homology classes of which constitute the
  basis~$S$. We see from equations~\eqref{eq:5},~\eqref{eq:6}
  and~\eqref{eq:7} that $1 - \Ta j$ kills all basis elements except
  $b_j$, and that $(1 - \Ta j) b_j = a_j$. Similarly, $1-\Tb j$ kills
  all basis elements except $a_j$, and $(1-\Tb j)a_j = -b_j$. If $u$
  is any cocycle, let $x_j$ denote the coefficient of $a_j$ in $u(\Ta
  j)$ and $y_j$ the coefficient of $b_j$ in $u(\Tb j)$. Adding the
  coboundary of the homology element
  \begin{equation*}
    \sum_{j=1}^g  y_j a_j - x_j b_j
  \end{equation*}
  to $u$ produces a cocycle with the required properties. It is clear
  that if $m$ is any non-zero homology element, there is some $\eta\in
  \mathcal{S}$ such that $(1-\tau_\eta)m$ contains a non-trivial
  $[\eta]$-component, proving the uniqueness claim.
\end{proof}
In a sense, with this proposition we have used up all the freedom
there is in the choice of cocycle representing a given cohomology
class.
\begin{proposition}
  \label{prop:2}
  The cohomology group $H^1(\Gamma, M)$ is a free abelian group of
  finite rank.
\end{proposition}
\begin{proof}
  We have just proved that there is a section $H^1(\Gamma, M)\to
  Z^1(\Gamma, M)$, which proves that there is no torsion. The claim
  about the finite rank follows from the fact that $\Gamma$ is
  finitely generated.
\end{proof}
% In fact, since the twists in the curves from $\mathcal{C}$ generate
% $\Gamma$, we get a very crude upper bound of $2g(3g-1)$ on
% the rank.

From now on, we will assume that $u$ is a cocycle which is adapted to
$\mathcal{S}$ in the sense of Proposition~\ref{prop:3}. We will now
proceed to determine other facts about $u$, using simple relations
between the twists in the curves from $\mathcal{C}$.

First, consider the braid relation
$\braidrel{\alpha_j}{\beta_j}$. Applying the cocycle condition, we obtain
\begin{equation}
  \label{eq:20}
  u(\Ta j) + \Ta j u(\Tb j) + \Ta j \Tb j u(\Ta j) = 
  u(\Tb j) + \Tb j u(\Ta j) + \Tb j \Ta j u(\Tb j).
\end{equation}
Since $u(\Tb j)$ is a linear combination of elements from $S-\{b_j\}$,
we have $\Ta j u(\Tb j) = u(\Tb j)$, and symmetrically, $\Tb j u(\Ta
j) = u(\Ta j)$. Hence two terms on either side of \eqref{eq:20}
cancel, and we are left with
\begin{equation}
  \label{eq:21}
  \Ta j u(\Ta j) = \Tb j u(\Tb j).
\end{equation}
By the normalization assumption, we have $\pi_j u(\Ta j) = y_j
b_j$ and $\pi_j u(\Tb j) = z_j a_j$ for some integers $y_j,
z_j$. Applying the projection $\pi_j$ to~\eqref{eq:21} and using that
$\pi_j$ commutes with the two twists, we obtain
\begin{equation*}
  -y_j a_j + y_j b_j = z_j a_j + z_j b_j
\end{equation*}
which implies $y_j = z_j = 0$. So we have
\begin{equation}
  \label{eq:24}
  \pi_j u(\Ta j) = \pi_j u(\Tb j) = 0 
\end{equation}
for all $1\leq j\leq g$. This in turn implies that $u(\Ta j)
= u(\Tb j)$ (apply $\pi_j'$ to~\eqref{eq:21} and use that it
annihilates the twists).

We record this consequence of our calculations so far.
\begin{lemma}
  \label{lem:3}
  For $g = 1$, $H^1(\Gamma, M)$ is trivial.
\end{lemma}

From now on we consider the case $g\geq 2$. For $k\not= j$, $\Ta j$
commutes with $\Ta k$ and $\Tb k$. Applying the cocycle condition to
these commutativity relations gives
\begin{align*}
  u(\Ta j) + \Ta j u(\Ta k) &=
  u(\Ta k) + \Ta k u(\Ta j) \\
  u(\Ta j) + \Ta j u(\Tb k) &=
  u(\Tb k)  + \Tb k u(\Ta j)
\end{align*}
and using the projection $\pi_k$ along with~\eqref{eq:24} we see that
\begin{align*}
  \pi_k u(\Ta j) &= \Ta k \pi_k u(\Ta j) \\
  \pi_k u(\Ta j) &= \Tb k \pi_k u(\Ta j).
\end{align*}
This clearly implies that $\pi_k u(\Ta j) = 0$. Since this
holds for any $k$, we obtain this important result.
\begin{proposition}
  \label{prop:1}
  A cocycle which is normalized in the sense of
  Proposition~\ref{prop:3} vanishes on each of the Dehn twists $\Ta
  j$, $\Tb j$.
\end{proposition}

The next lemma may sound a bit cryptic, but the text following the
proof should make its usefulness clear.
\begin{lemma}
  \label{lem:4}
  If $\sigma$ is a simple closed curve disjoint from some $\eta\in
  \mathcal{S}$, then the coefficient of $[\iota \eta] = \iota[\eta]$
  in $u(\tau_\sigma)$ is~$0$, where $u(\tau_\sigma)$ is written in
  terms of the basis $S$.
\end{lemma}
\begin{proof}
  Since $\sigma$ and $\eta$ are disjoint, the associated twists
  commute. Applying the cocycle condition and the vanishing of
  $u(\tau_\eta)$ this yields
  \begin{equation*}
    u(\tau_\sigma) = \tau_\eta u(\tau_\sigma).
  \end{equation*}
  Since $\tau_\eta (\iota[\eta]) = \iota [\eta] \pm [\eta]$ and
  $\tau_\eta$ acts as the identity on all other basis element, this is
  only possible if $u(\tau_\sigma)$ does not have a
  $\iota[\eta]$-component.
\end{proof}
Notice, for example, that this implies that $u(\Tg j)$ is a linear
combination of $a_{j-1}$ and $a_{j}$, since $\beta_{j-1}$ and
$\beta_{j}$ are the only curves from $\mathcal{S}$ which $\gamma_j$
intersect. In fact we have:
\begin{lemma}
  \label{lem:5}
  There are integers $q_j$, $j=2, \ldots, g$, such that $u(\Tg j) =
  q_j a_{j-1} - q_j a_j = q_j c_j$.
\end{lemma}
The case $j=1$ is only omitted because $a_0$ is not defined; we
clearly have $u(\Tg 1) = u(\Ta 1) = 0 = 0c_1$.
\begin{proof}
  The braid relation between $\Tb j$ and $\Tg j$ yields
  \begin{equation}
    \label{eq:26}
    \Tb j u(\Tg j) = u(\Tg j) + \Tg j \Tb j u(\Tg j)
  \end{equation}
  after applying the cocycle condition and the vanishing of $u(\Tb
  j)$. % Applying $\pi_k$ to both sides of \eqref{eq:26} for some
  % $k\not= j-1, j$, and using that $\pi_k \Tb j = \pi_k \Tg j = \pi_k$,
  % we see that $\pi_k u(\Tg j) = 0$, so that $u(\Tg j) \in M_{j-1}
  % \oplus M_j$. Write
  % \begin{equation*}
  %   u(\Tg j) = p_j a_{j-1} + r_j b_{j-1} + q_j a_j + s_j b_j
  % \end{equation*}
  % for integers $p_j, r_j, q_j, s_j$.
  Using Lemma~\ref{lem:4}, we may write $u(\Tg j) = q_j a_{j-1} + p_j
  a_j$ for some integers $q_j, p_j$. Then the left-hand side
  of~\eqref{eq:26} is
  % \begin{equation*}
  %   p_j a_j + (r_j + p_j) b_j + q_j a_{j+1} + s_j b_{j+1}
  % \end{equation*}
  \begin{equation*}
    q_j a_{j-1} + p_j a_j + p_j b_j
  \end{equation*}
  while the right-hand side is
  % \begin{equation*}
  %   (p_j-r_j+s_j) a_j + (2r_j+p_j) b_j + (2q_j + r_j + p_j - s_j) a_{j+1} + 2s_j b_{j+1}.
  % \end{equation*}
  \begin{equation*}
    (q_j a_{j-1} + p_j a_j) + (q_j a_{j-1} + p_j a_j + p_j b_j + p_j
    a_{j-1} - p_j a_j) = (2q_j + p_j) a_{j-1} + p_j a_j + p_j b_j.
  \end{equation*}
  From this it follows that %$r_j = s_j = 0$, and that
  $p_j = -q_j$, so we do indeed have $u(\Tg j) = q_j (a_{j-1} - a_j) =
  q_j c_j$.
\end{proof}
Thus $u$ is completely determined by the $g-1$ integers $q_2, \ldots,
q_g$. In particular, the rank of $H^1(\Gamma, H)$ is now bounded above
by $g-1$.

We now adapt the cocycle to the set $\mathcal{S}'$:
\begin{proposition}
  \label{prop:5}
  Let $u$ be a cocycle adapted to $\mathcal{S}$ in the sense of
  Proposition~\ref{prop:3}. Then $u$ is cohomologous to a cocycle,
  again denoted $u$, satisfying $u(\Tb j) = u(\Tg j) = 0$ for $j=1,
  \ldots, g$.
\end{proposition}
Note that this new $u$ may no longer vanish on the twists $\Ta j$ for
$j\geq 2$.
\begin{proof}
  Since $(1-\Tb j)b_k = 0$ for all $j,k$, and since $u$ already
  vanishes on $\Tb j$ for all $j$, adding the coboundary of any
  homology element which is a linear combination of the $b_j$
  preserves this property.

  Let $q_j$ denote the integers such that $u(\Tg j) = q_j c_j$. Put
  $r_1 = 0$ and $r_j = r_{j-1} + q_j$ for $j > 1$. Then
  \begin{equation*}
    (1-\Tg k) \sum_{j=1}^g r_j b_j = r_{k-1} c_k - r_k c_k = -q_k c_k
  \end{equation*}
  so adding the coboundary of $\sum_{j=1}^g r_j b_j$ produces a
  cocycle with the required properties.
\end{proof}

With these preparations, we can finally compute $H^1(\Gamma, M)$.
\begin{theorem}
  \label{thm:1}
  For $g\geq 3$, the cohomology group $H^1(\Gamma, M)$ is isomorphic
  to $\setZ$.
\end{theorem}
\begin{proof}
  By Theorem~\ref{thm:2}, the mapping class group is generated by the
  twists in the curves from $\mathcal{S}'\cup \{\alpha_2\}$. Since a
  cocycle which is adapted to $\mathcal{S}'$ vanishes on each of the
  twists in these curves, we see that such a cocycle is completely
  determined by the element $u(\Ta 2) \in M$. In fact, we can say
  even more: Before we perform the adaptation to $\mathcal{S}'$, we
  have $u(\Ta 2) = 0$, so it follows by construction that
  \begin{equation*}
    u(\Ta 2) = (1-\Ta 2)\bigl( \sum_j r_j b_j \bigr) = r_2 a_2.
  \end{equation*}
  Hence the cocycle is determined by the single integer $r_2 = q_2$,
  and $H^1(\Gamma, M)$ is thus either trivial or isomorphic to
  $\setZ$.

  To see that it is not trivial, we must prove that putting
  $u(\tau_\eta) = 0$ for $\eta\in \mathcal{S}'$ and $u(\Ta 2) = a_2$
  defines a cocycle. To do this, we need to check each of the
  relations in Wajnryb's presentation.

  The disjointness and braid relations are only interesting if one of
  the curves is $\alpha_2$. But if $\eta\in \mathcal{S}'$ is disjoint
  from $\alpha_2$ (that is, $\eta$ is not $\beta_2$), we have
  \begin{equation*}
    u(\Ta 2 \tau_\eta) = u(\Ta 2) + \Ta 2 u(\tau_\eta) = a_2 =
    \tau_\eta a_2 = u(\tau_\eta) + \tau_\eta u(\Ta 2) = u(\tau_\eta
    \Ta 2).
  \end{equation*}
  We also have
  \begin{align*}
    u(\Ta 2 \Tb 2 \Ta 2) &= a_2 + \Ta 2 \Tb 2 a_2 = a_2 + b_2 \\
    u(\Tb 2 \Ta 2 \Tb 2) &= \Tb 2 a_2 = a_2 + b_2.
  \end{align*}

  Applying $u$ to the left-hand side of~\eqref{eq:25} and expanding
  via the cocycle condition clearly gives $0$. On the right-hand side,
  we get
  \begin{align*}
    u(\Ta 2) + u(w \Ta 2 w\inv) = a_2 + (1-w\Ta 2w\inv) u(w) + w a_2
  \end{align*}
  by~\eqref{eq:31}. Clearly $u(w)$ vanishes, and it is
  straight-forward to check that $w a_2 = -a_2$ using
  equations~\eqref{eq:8} and~\eqref{eq:22}.

  Finally, we compute the values of $u$ on the various auxilliary
  words occuring in~\eqref{eq:28}.
  \begin{align*}
    u(x_1) &= u(w_1\inv \Ta 2 w_1) = (1-w_1\inv)u(w_1\inv) + w_1\inv
    u(\Ta 2) = w_1\inv a_2 \\
    &= \Tb 2\inv \Tg 2\inv \Tg 3\inv (a_2 - b_2) = \Tb 2\inv \Tg 2\inv
    (-b_2 + a_3) = \Tb 2 (-b_2 + a_3 + a_1 - a_2) \\
    &= a_1 - a_2 + a_3
  \end{align*}
  \begin{align*} 
    u(x_2) &= w_2\inv u(x_1) = \Tb 1\inv \Tg 1\inv \Tg
    2\inv(a_1-b_1-a_2+a_3) = \Tb 1\inv \Tg 1\inv(a_3 - b_1) = \Tb
    1\inv(-a_1-b_1+a_3) \\
    &= a_3 - a_1
  \end{align*}
  \begin{align*}
    u(x_3) &= w_3\inv u(x_1) = \Tg 3\inv \Tb 3\inv (a_1 - a_2 + a_3) =
    \Tg 3\inv(a_1 - a_2 + a_3 - b_3) \\
    &= a_1 - b_3
  \end{align*}
  \begin{align*}
    u(w_4) = \Tb 3 \Tg 3 \Tb 2 \Tg 2 \Tb 1 u(x_3) = b_1 - a_2 + b_2 +
    2a_3 + b_3
  \end{align*}
  \begin{align*}
    u(x_4) = (1-w_4 \Ta 2 w_4\inv) u(w_4) + w_4 u(\Ta 2) = 2a_3
  \end{align*}
  Finally, the value of $u$ on the left-hand side of~\eqref{eq:28} is
  \begin{align*}
    u(\Ta 2 x_2 x_1) &= u(\Ta 2) + \Ta 2 u(x_2) + \Ta 2 x_2 u(x_1) = a_2 + (-a_1 + a_3) + (a_1-a_2+a_3)
    \\ &= 2a_3,
  \end{align*}
  whereas the value on the right-hand side is
  \begin{align*}
    u(\Tg 1\Tg 2\Tg 3 x_4) = \Tg1 \Tg2 \Tg3 u(x_4) = 2a_3.
  \end{align*}

  Hence the map $\{\tau_\eta \mid \eta \in \mathcal{S}' \cup
  \{\alpha_2\} \}\to M$ defined by $\Ta 2\mapsto a_2$ and
  $\tau_\eta\mapsto 0$ for $\eta\in \mathcal{S}'$ extends to a cocycle
  defined on $\Gamma$, and this cocycle represents a generator for the
  cohomology group $H^1(\Gamma, M)$.
\end{proof}

\section{Final remarks}
\label{sec:final-remarks}

\subsection{Other computations}
\label{sec:other-computations}

Earle~\cite{MR0499328} constructed a cocycle $\psi\colon \Gamma\to
H_1(\Sigma, \setR)$ such that $(2g-2)\psi$ has values in $H_1(\Sigma,
\setZ)$. Later Morita~\cite{MR1030850} proved that $H^1(\Gamma,
H_1(\Sigma, \setZ))\isom \setZ$ using a combinatorial
approach. Recently, Kuno~\cite{MR2461871} has computed Earle's cocycle
in terms of Morita's. Satoh~\cite{MR2184815} has computed the first
homology and cohomology of the automorphism and outer automorphism
groups of a free group with coefficients in the abelianization of the
free group. This list of references is of course by no means
exhaustive.

\subsection{Surface with a marked point}
\label{sec:surface-with-marked}

Instead of a surface with boundary, one could consider a closed
surface $\Sigma_*$ with a marked point (or equivalently, a closed
surface minus a point). Denote the mapping class group of $\Sigma_*$
by $\Gamma_*$. 
\begin{proposition}
  \label{prop:6}
  The natural homomorphism $\Gamma \to \Gamma_*$ induces an
  isomorphism $H^1(\Gamma, M) \isom H^1(\Gamma_*, M)$, where $M =
  H_1(\Sigma) \isom H_1(\Sigma_*)$.
\end{proposition}
\begin{proof}
  The map $\Gamma\to \Gamma_*$ is obtained by gluing a disc with a
  marked point to the boundary of $\Sigma$ and extending a
  representative of a mapping class $f\in \Gamma$ by the
  identity. Clearly the inclusion $\Sigma\to \Sigma_*$ is an
  isomorphism on $H_1$, and this is equivariant with respect to the
  homomorphism $\Gamma\to \Gamma_*$. We identify the two homology
  groups via this isomorphism.

  Now, $\Gamma\to \Gamma_*$ is surjective, and the kernel is the
  infinite cyclic group $\langle \tau_\boundary\rangle$ generated by
  the twist in the boundary of $\Sigma$
  \cite[Proposition~3.19]{MCGprimer}. Note that $\tau_\boundary$ acts
  trivially on $M$, so $M^{\langle \tau_\boundary\rangle} = M$ and
  $H^1(\langle \tau_\boundary\rangle, M) = \Hom(\langle
  \tau_\boundary\rangle, M)$. Applying the exact
  sequence~\eqref{eq:30} we obtain
  \begin{equation}
    \label{eq:32}
    0\to H^1(\Gamma_*, M) \to H^1(\Gamma, M) \to H^1(\langle
    \tau_\boundary\rangle, M)^\Gamma = 0,
  \end{equation}
  since a homomorphism from $\langle\tau_\boundary\rangle$ is simply
  an element of $M$, and $0$ is the only $\Gamma$-invariant element of
  $H_1(\Sigma)$.
\end{proof}

%%% Now we print the bibliography

%%% The biblatex + biber way
\ifarXiv
%%% The old-fashioned bibtex way
\bibliographystyle{plain}
\bibliography{rvbibinfo}

\else
\printbibliography
\fi

\end{document}

%%% Local Variables: 
%%% mode: latex
%%% TeX-master: t
%%% TeX-command-BibTeX: "Biber"
%%% End: 